\documentclass[sn-mathphys-num]{sn-jnl}% Math and Physical Sciences Numbered Reference Style 
\usepackage{graphicx}%
\usepackage{multirow}%
\usepackage{amsmath,amssymb,amsfonts}%
\usepackage{amsthm}%
\usepackage{mathrsfs}%
\usepackage[title]{appendix}%
\usepackage{xcolor}%
\usepackage{textcomp}%
\usepackage{manyfoot}%
\usepackage{booktabs}%
\usepackage{algorithm}%
\usepackage{algorithmicx}%
\usepackage{algpseudocode}%
\usepackage{listings}%
\newcommand{\U}{\mathcal {U}}

\newcommand{\fin}[0]{{\subset\!\!\!\subset}}

\theoremstyle{thmstyleone}%
\newtheorem{theorem}{Theorem}%  meant for continuous numbers
\newtheorem{proposition}[theorem]{Proposition}
\newtheorem{lemma}[theorem]{Lemma}
\newtheorem{corollary}[theorem]{Corollary}

\theoremstyle{definition}
\newtheorem{definition}{Definition}%
\newtheorem{remark}{Remark}%

\begin{document}
	
	\title[Topological asymptotic dimension]{Topological asymptotic dimension}
	
	
	\author{\fnm{Massoud} \sur{Amini}\email{mamini@modares.ac.ir, mamini@ipm.ir}\footnote{The  author was supported by a grant from the Iran National Science Foundation (INSF, Grant No. 4027151).}}

	\affil{\orgdiv{Faculty of Mathematical Sciences}, \orgname{Tarbiat Modares University}, \city{Tehran}, \postcode{14115-134}, \country{Iran}}

	
	\abstract{We initiate a study of asymptotic dimension for locally compact groups. This notion extends the existing invariant for discrete groups and is shown to be finite for a large class of residually compact groups. Along the way, the notion of Hirsch length is extended to topological groups and classical results of Hirsch and Malcev are extended using a topological version of the Poincar\'{e} lemma. We show that polycyclic-by-compact groups and compactly generated, topologically virtually nilpotent groups are residually compact, and that compactly generated nilpotent groups are polycyclic-by-compact. We prove that for compactly generated, solvable-by-compact
		groups the  asymptotic dimension is majorized by the Hirsch length, and equality holds for polycyclic-by-compact groups. We extend the class of elementary amenable  groups beyond the discrete case and show that topologically elementary amenable groups with finite Hirsch length have finite asymptotic dimension. We prove that a topologically elementary amenable group of finite Hirsch length with no nontrivial locally elliptic normal closed subgroup is solvable-by-compact. Finally, we show that a totally disconnected, locally compact, second countable [SIN]-group has finite asymptotic dimension, if all of its discrete quotients are so.
	}

	
	

	\maketitle

\section{Introduction} \label{int}
The study of finitely generated discrete groups has witnessed two historical developments. The first is celebrated work of Hirsch on the structure theory of infinite solvable groups in the 30's, where he initiated an invariant (now called the Hirsch length) and showed that many structural arguments could be based on induction on this length, quite like what is done for nilpotent groups using the niloptency class \cite{h}. The second is pioneering work of Gromov in the 90's on asymptotic invariants of infinite groups \cite{g}. These two developments are interestingly related, as the Hirsch length is known to majorize the Gromov asymptotic dimension in many cases \cite{bcl}, \cite{ds}. 

Little is known about both invariants beyond the discrete case. The Hirsch length has technical difficulties when one deals with topological groups. The main source of complication is the failure of the second isomorphism theorem for topological groups which makes the usual definition not well-defined. The Gromov invariant has a better situation, as it is defined as soon as we have a proper, translation invariant metric 
which generates  topology of the group (so-called a plig metric \cite{hp}). However, it is not easy to control the invariant not to blow up. The key tool to maintain such a control is to relate the two invariants, as in the discrete case. This is the main objective of the current paper. 

Here instead of finitely generated discrete groups one has to work with compactly generated locally compact groups (usually second countable, to make sure a plig metric exists). There are several topological complications to be addressed, but the outcome is relatively satisfactory: the invariants are well-defined and finite in many interesting cases and interrelate to shed a light on the structure theory of certain natural classes of topological groups. 

In Section \ref{Box} we introduce box spaces and define asymptotic dimension of locally compact (second countable) groups and show that the invariant is finite for large classes of residually compact groups. In Section \ref{tvn} we define the Hirsch length for topological groups and prove the Hirsch formula, which is critical for inductive arguments based on the Hirsch length. This section also deals with a natural extension of virtually nilpotent groups in the topological realm, which are shown to be residually compact and so of finite asymptotic dimension. Same is done in Section \ref{ea} for topologically elementary amenable groups. A small final epilogue in Section \ref{td} is devoted to finiteness of the asymptotic dimension of totally disconnected, locally compact, second countable groups.    

\section{Box spaces and asymptotic dimension} \label{Box}

In this section we extend the notion of Box space and its asymptotic dimension from the class of discrete residually finite groups to that of locally compact residually compact groups. This is not a trivial task and we have to overcome several topological technicalities.

Throughout the rest of this section,  $G$ is a locally compact, $\sigma$-compact, Hausdorff group and $H\leq G$ is a closed subgroup. We denote both the identity element and trivial subgroup of $G$ by 1. For a complex or real valued function $f$ on $G$, the left translate of $f$ is defined by $\ell_gf(h)=f(g^{-1}h)$, for $g,h\in G$. We use the standard notations $H^g:=gHg^{-1}$ and $g^H:=\{hgh^{-1}: h\in H\}$, for $g\in G$. We also denote a finite subset $E$ of $G$ by $E\fin G$ and a normal subgroup $H$ of $G$ by $H\trianglelefteq G$.

Recall that $H$ is called {\it cocompact} if there is a compact subset $K\subseteq G$ with $G=KH$. For a locally compact group $G$, a closed subgroup $H$ is cocompact iff the quotient space $G/H$ is compact (combine \cite[Lemma 2.C.9]{ch}, with the idea of the proof of \cite[Proposition 18]{th}). This is the topological analog of subgroups of finite index.

\vspace{.3cm}
\begin{definition}\label{rc}
 A {\it residually compact} (resp.,  {\it residually finite}) approximation ($RC$-app.; resp., $RF$-app.) of $G$ is a decreasing sequence  of  cocompact closed (resp., open) subgroups $G_n\leq G$ with $\bigcap_{n\geq 1} G_n=1$. If $G$ has an $RC$-app. (resp., $RF$-app.), we say that $G$ is  {\it residually compact} (resp., {\it residually finite}) ($RC$; res., $RF$).
\end{definition}

\vspace{.3cm}
All compact groups are $RC$. On the other hand, an uncountable compact group is $RC$ but not $RF$. Note that for an $RF$-app. $(G_n)$, each quotient space $G/G_n$ is both compact and discrete, and so finite, that is, $G_n$ is of finite index.

 We shall introduce an uncountable version of the above notion in Section \ref{ea} which also applies to non compactly generated groups. The two notions are equivalent for locally compact, $\sigma$-compact  (or equivalently, compactly generated) groups. 

The equivalence of conditions $(i)$-$(iii)$ in the next definition could be proved by an argument verbatim to that of \cite[Lemma 23.7]{swz}. 

\vspace{.3cm}
\begin{definition}\label{rrc}
 An $RC$-app. or $RF$-app. $(G_n)$ of $G$ is called {\it regular} if it satisfies any of the following equivalent conditions:

 $(i)$ for each $g\neq 1$ there is $n\geq 1$ with $G_n\bigcap g^GG_n=\emptyset$,

 $(ii)$ $\bigcap_{n\geq 1}\bigcup_{g\in G} gG_ng^{-1}=\{1\}$;

 $(iii)$ for each $E\fin G$ there is $n\geq 1$ such that the quotient map: $G\to G/G_n$ is injective on $Eg$, for each $g\in G$.
\end{definition}

\vspace{.3cm}
The above equivalent conditions trivially hold when each $G_n$ is normal. When $G$ is discrete, in a residually finite approximation ($RF$-app.) $(G_n)$, the subgroups $G_n$ are of finite index. In this case, $RF$-approximations   may always be chosen to be regular, as we may replace $G_n$'s with smaller normal subgroups of finite index via the Poincar\'{e} lemma: For a discrete group $G$ and subgroup $H\leq G$ of finite index $n:=[G:H]$, there is a normal subgroup  $N\trianglelefteq G$ with $N\leq H$ and $[G:N]|\leq n!$. The standard proof of Poincar\'{e} lemma uses the fact that $|{\rm Sym}(G/H)|=[G:H]!$ and that for $N:=\bigcap_{g\in G} H^g$, a copy of $G/N$  embeds in ${\rm Sym}(G/H)$. This argument breaks down in the topological version, since while the above $N$ is a closed normal subgroup contained in $H$ and $G/N$  continuously embeds in ${\rm Iso}(G/H)$, the group of isometries of the compact metric space $G/H$, the copy of $G/N$ fails to be closed in ${\rm Iso}(G/H)$. For instance, $G=SL(2, \mathbb R)$ has many cocompact subgroups (cocompact lattices, as well as the upper triangular subgroup), but no proper normal cocompact subgroup. However the situation could be saved in the topological case by relaxing the cocompactness to a local version. 

\vspace{.3cm}
\begin{definition}\label{lcc}
	A closed subgroup $H$ is called {\it locally cocompact} if $G/H$ has zero asymptotic dimension with respect to the  quotient metric induced by an adapted metric on $G$.  
\end{definition}

\vspace{.3cm}
A few comments are in order: First that a closed normal subgroup $N$ is locally cocompact iff the quotient group $G/N$ is locally elliptic (i.e., its closed subgroups generated by compact subsets are compact). Since local ellipticity is indeed a local version of compactness in coarse geometric sense (rather than topological sense), this justifies the term used here. Second, $G$ is throughout assumed to be $\sigma$-compact, and so it always has an adapted metric \cite[Proposition 4.A.2]{ch}, and each two such metrics are coarsely equivalent \cite[Corollary 4.A.6(2)]{ch}, and the same holds for the induced metrics on the quotient space, and so the above notion is well defined. Third, the quotient metric is not necessarily $G$-invariant (though the quotient space $G/H$ is a disjoint union of $\sigma$-compact spaces--namely, the cosets of $H$--and so paracompact \cite[Theorem 7.3]{d}, the action of $G$ on $G/H$ is not proper--unless, for instance $H$ is compact--and so results such as \cite[Proposition 4.C.8(1)]{ch} do not apply).       

\vspace{.3cm}
\begin{lemma}  [Poincar\'{e}] \label{pl}
For each cocompact closed  subgroup $H\leq G$, if there is an adapted metric on $G$ inducing a $G$-invariant compatible metric on $G/H$, then there is a locally cocompact closed normal subgroup $N\trianglelefteq G$ with $N\leq H$.
\end{lemma}
\begin{proof}
Choose an adapted metric on $G$ with $G$-invariant quotient metric on $G/H$. Consider the closed normal subgroup $N:=\bigcap_{g\in G} H^g$. Then $G/H$ is a compact metric space on which $G/N$ acts from left by 
$$xN\cdot gH=xgH \ (x,g\in G)$$
This is well defined as $N$ is normal and isometric as the metric on $G/H$ is left translation invariant. In particular, $G/N$ continuously embeds in the compact metric group ${\rm Iso}(G/H)$. Though the image is not closed, but $G/N$ inherits an adapted metric from the metric group ${\rm Iso}(G/H)$, which has zero asymptotic dimension (as it is compact), thus asdim$(G/N)$=0, and so it is locally elliptic by \cite[Proposition 4.D.4]{ch}. 
\end{proof}

\vspace{.3cm}
Note that by the above proof, an open cocompact subgroup $H$ contains a $\mathcal G_\delta$ (countable intersection of open sets) locally cocompact normal subgroup $N$. This is important, since if $N$ is normal and $\mathcal G_\delta$, then for each closed subgroup $K$, $NK$ is locally compact (see Lemma \ref{sec}).

\vspace{.3cm}
\begin{corollary}\label{rca}
	If $G$ is a residually compact group with an $RC$-app. (resp., $RF$-app.) $\{G_n\}$ such that each quotient space $G/G_n$ has a $G$-invariant compatible metric, then $G$ also has a regular $RC$-app. (resp., $RF$-app.).
\end{corollary}

\vspace{.3cm}
The assumption that the quotient space $G/H$ has a $G$-invariant metric could be quite restrictive, for instance, in the above example where $G=SL(2,\mathbb R)$ and $H$ is a cocompact lattice in $G$, there is no invariant metric on the orbit space $G/H$, since otherwise, by the above lemma $G$ must have a locally cocompact normal subgroup $N$, but then since $G=SL(2,\mathbb R)$ is compactly generated, and so is $G/N$ \cite[Proposition 2.C.8(4)]{ch}. This forces $G/N$ to be compact, i.e., $N$ must be a cocompact normal subgroup, which is not possible, as $N$ is contained in a cocompact lattice. On the other hand, this assumption is automatic for $G/G_n$ if we work with an $RF$-approximation. 

The problem of existence of a $G$-invariant compatible (i.e., giving back the topology) metric on a Hausdorff topological group $G$ was independently solved by Birkhoff \cite{bi} and Kakutani \cite{k}: a topological group is metrizable if and only if it is Hausdorff and first countable, and in this case, the metric can be taken to be $G$-invariant (c.f., \cite[section 1.22]{mz}). If we ask for a $G$-invariant compatible metric which is also proper (i.e., the
balls are relatively compact), the problem is solved by Struble \cite{st}: a locally compact
group $G$ has a proper $G$-invariant compatible metric if and only if it is second
countable. 
The same questions on coset spaces are explored by Delaroche in \cite{de}. Let $H$ be
a closed subgroup of a Hausdorff topological group $G$. When $G$ is metrizable,
the quotient space $G/H$ has a quotient metric, however,
even if $G$ is a second countable, there does not always
exist a $G$-invariant compatible metric on $G/H$ (see, \cite[Example 2.3]{de}).  When $H$ is an open subgroup of a second countable group $G$,  a $G$-invariant compatible metric exists on $G/H$ iff $H$ is {\it almost normal} (i.e., for every
finite subset $F$ of $G$ there exists a finite subset
$E$ of G such that $HF \subseteq EH$) \cite[Theorem 2.15]{de}, but this is not sufficient for closed subgroups \cite[Example 2.14]{de}.

A closed subgroup $H\leq G$ is called a {\it characteristic subgroup} if it is invariant under $Aut(G)$ consisting of  continuous automorphisms of $G$. Since inner automorphisms are automatically continuous, each  characteristic subgroup is normal (but the converse is not correct even in discrete groups, as for instance $G\times 1$ is a normal subgroup of $G\times G$, but is not preserved under the flip automorphism).

One could use the same argument to prove (an extension  of) Poincar\'{e} lemma for the characteristic subgroup  $N:=\bigcap_{\sigma\in Aut(G)} \sigma(H)$.  We record this result in the next lemma for future use.

\vspace{.3cm}
\begin{lemma}  \label{char}
For each cocompact closed  subgroup $H\leq G$, if there is an adapted metric on $G$ inducing a $G$-invariant compatible metric on $G/H$, then there is a locally cocompact closed characteristic subgroup $N\trianglelefteq G$ with $N\leq H$. In particular, if $G$ is abelian, every closed cocompact subgroup $H\leq G$ contains a locally cocompact closed characteristic subgroup $N$. 
\end{lemma}

\vspace{.3cm}
The fact that in the above lemma gives the existence of locally cocompact closed characteristic subgroups for locally compact, $\sigma$-compact, abelian groups seems annoying, but in fact it is quite useful, since for virtually nilpotent groups (which are the main objects of interest in this section), one could reduce to the case of abelian groups (using an inductive argument) and use the above result.

Let  $\tau=(H_n)$ and $\sigma=(G_n)$ be decreasing sequences of cocompact subgroups of $G$, then we write $\tau\precsim\sigma$ if for each $n\geq 1$ there is $m\geq 1$ with $G_m\subseteq H_n$, and write $\tau\sim\sigma$ if $\tau\precsim\sigma$ and $\sigma\precsim\tau$. The set $\Lambda(G)$ of all $RC$-apps. of $G$ is upward closed w.r.t. $\precsim$. A maximum element (if any) is called a {\it dominating} $RC$-app. It follows from the above lemma that every compact subgroup $H\leq G$ is part of a regular $RC$-app., thus a regular $RC$-app. $(G_n)$ is dominating iff for each compact subgroup $H\leq G$, there is $n\geq 1$ with $G_n\leq H$. Dominating $RC$-apps. are important as they give the lowest possible asymptotic dimension (among all $RC$-apps.). The corresponding box space is denoted by $\Box_s G$ and is called the {\it standard} box space. Note that (as in the discrete case)  ${\rm asdim}(\Box_s G)$
is well defined (c.f.  \cite[Remark 3.21]{swz}.)

While each finitely generated $RF$ group has a dominating regular $RF$-app. (consisting of normal subgroups), this may fail for compactly generated $RC$ groups and  regular $RC$-apps. This is because the former is a consequence of the fact that finitely generated $RF$ groups have at most countably many normal subgroups of finite index, while compactly generated $RC$ groups may have uncountably many normal cocompact subgroups (Of course, for $RC$ groups with at most countably many normal cocompact subgroups, one could always find a  dominating regular $RC$-app.)

Next let $G$ act (from right) by isometries on a metric space $(X,d)$. The
action is {\it locally bounded} if $K\cdot x$ is bounded for every $x \in X$ and every compact
set $K\subseteq G$ and it is  {\it bounded} if every orbit is bounded. The action is (metrically) {\it proper} if $d(x\cdot g, x)\to\infty$, for every $x\in X$ as $g\to \infty$ in $G$. Let  $\pi:X\to X/G$ be the corresponding quotient map. Then
$$D(\pi(x),\pi(y)):=\inf_{g,h\in G} d(x\cdot g, y\cdot h)$$
defines a metric on the quotient space $X/G$.

The next lemma is well-known.

\vspace{.3cm}
\begin{lemma}\label{pd}
When the action is isometric, properly-discontinuous and cobounded, the metric $D$ induces the quotient topology on $X/G$.
\end{lemma}

\noindent A {\it length function} on $G$ is a  map $\ell: G\to \mathbb R^+$ which  satisfies

$(i)$ $\ell(gh)\leq \ell(g)+\ell(h)$;

$(ii)$ $\ell(g^{-1})=\ell(g)$;

$(iii)$ $\ell(g)=0\Leftrightarrow g=1$,

\noindent for each $g,h\in G$.

\vspace{.3cm}
If $G$ admits a
locally bounded action by isometries on some metric space $(X, d)$, then for every $x\in X$ the map $g\mapsto d(x\cdot g, x)$ is a length function on $G$, which is bounded
on compact subsets. It is known that $G$ admits
a proper length function (bounded
on compact subsets) if and only if $G$ is $\sigma$-compact \cite{c}. Also every
locally bounded action of $G$ by isometries is either bounded or proper, iff every length function on $G$ is either bounded or proper \cite{c} (This is called property $PL$ by  Cornulier, see also \cite{tv}).

Length functions also naturally induce proper, left translation  invariant metrics which generate the  topology on $G$. Such a metric is called a plig metric by Haagerup and  Przybyszewska \cite{hp}.
It is  easy to show that if a locally compact group $G$ admits a plig metric, then $G$
is second countable. Moreover, any two plig metrics  on a second countable group $G$ are coarsely equivalent, in the sense of Roe \cite{r2}.
Lubotzky, Moser and Ragunathan showed  that every compactly
generated second countable group has a plig metric \cite{lmr} (same holds for countable discrete groups by a result of  Tu \cite{t}). In the most general case, Haagerup and  Przybyszewska showed that every locally compact, second countable group
admits a plig metric (in a way that the balls have
exponential  growth w.r.t. the Haar measure). They also showed that the existence of a plig metric on $G$ implies that $G$ has bounded geometry.

Next let $G$ be a second-countable $RC$ group with  a plig metric $d_G$ and for an $RC$-app. $\sigma=(G_n)$, let $d_n$ be the corresponding quotient metric on $G/G_n$. We define the {\it box space} of the $RC$-app. $\sigma$ by
$$\Box_\sigma G:=\bigsqcup_{n\geq 1} (G/G_n, d_n),$$
with box metric $d_B$ which restricts to $d_n$ on  $G/G_n$ and $d_B(G/G_n, G/G_m)\to \infty$ as $n+m\to\infty$. When $d_G$ is induced by a length function $\ell$ on $G$, one such box metric is given by
$$d_B(gG_n, hG_m):=\ell(g)+\ell(h)+n+m,$$
for $g,h\in G$ and $n,m\geq 1$.

The following lemma is proved as in the discrete case \cite[Lemma 3.13]{swz} (Note that \cite[Lemma 2.11]{swz} and  \cite[Lemma 4.3.5]{ny} used in the proof of \cite[Lemma 3.13]{swz} are stated for general metric spaces and could be used in our setting).

\vspace{.3cm}
\begin{lemma} \label{basic}
Let $G$ be $RF$ with a given proper right invariant metric $d$ and left Haar measures $\lambda$. Let $G$ act on itself by right multiplication. For $d\geq 0$ and a given $RF$-app. $\sigma=(G_n)$, the following are equivalent:

$(i)$ $\sigma$ is regular and ${\rm asdim}(\Box_\sigma G)\leq d$;

$(ii)$ for each $R>0$ there is $n\geq 1$ and a uniformly bounded cover $\U=\{U_0,\cdots, U_d\}$ of $G$ consisting of mutually disjoint $G_n$-invariant sets, with Lebesgue number at most $R$;

$(iii)$ for each $\varepsilon>0$ and each $E\fin G$ there is $n\geq 1$ and  functions $f_i: G\to [0,1]$ with compact support $K_i$, with $K_i\cap K_ig=\varnothing$, for $g\neq 1$, satisfying $\|f_i-\ell_gf_i\|_\infty\leq \varepsilon$ and
$$\sum_{i=0}^{d}\sum_{h\in G_n}f_i(gh)=1\ \ (g\in G),$$
for $0\leq i\leq d$.
\end{lemma}

\vspace{.3cm}
\noindent The family $\{f_i\}^d_{i=0}$ is called a {\it system of decay} for $\sigma$. 
Next, we need the following well known version of the second isomorphism theorem for topological groups \cite[Theorem 5.33]{hr}.

\vspace{.3cm}
\begin{lemma}\label{sec}
Let $G$ be a topological group and $H,K$ be closed subgroups of $G$ with $H$ normal and $K$ locally compact and $\sigma$-compact such that $KH$ is locally 
compact. Then $KH/H$ and $K/(K\cap H)$ are canonically isomorphic as  
topological groups. 
\end{lemma}

\vspace{.3cm}
It is known that intersection of cocompact subgroups may fail to be cocompact  (for instance, for an irrational real number $\alpha$, the closed cocompact subgroups $\mathbb Z$ and $\mathbb Z\alpha$ of $\mathbb R$ intersect trivially). On the positive side, we have the following result when one of the subgroups is open.  

\vspace{.3cm} 
\begin{lemma}\label{intersection}
Let $G$ be a locally compact $\sigma$-compact group, $H$ be a normal closed subgroup of $G$ and $K$ be a closed cocompact subgroup of $G$, then 

$(i)$ if $K$ or $H$ is open, then $K\cap H$ is a closed cocompact subgroup of $H$,

$(ii)$ if moreover $H$ is cocompact,  then $K\cap H$ is a closed cocompact subgroup of $G$.
\end{lemma}
\begin{proof}
$(i)$ Since $K$ is closed, it is locally compact and $\sigma$-compact. If $K$ or  $H$ is open, $KH$ is an open subgroup, and so closed and locally compact. Now $KH/K$ is a closed subgroup of the compact group $G/K$, thus it is compact. It follows from Lemma \ref{sec} that $H/(K\cap H)$ is compact. 

$(ii)$ When $H$ is moreover cocompact we may repeat the above argument to see that $KH/H$ is compact, and so is $K/(K\cap H)$, again by Lemma \ref{sec}. Choose a compact subset $J\subseteq K$ with $K=J(K\cap H)$. Since $K$ is cocompact, we may also choose a compact subset $L\subseteq G$ such that $G=LK$. It follows that $G=(JL)(K\cap H)$, thus $K\cap H$ is cocompact, as $JL$ is compact. 
\end{proof}

A slightly more general version of the above lemma holds: if $H=\bigcap_{n\geq 1} H_n$ is a countable intersection of open normal subgroups and $K$ is a closed cocompact subgroup, then $K\cap H$ is a closed cocompact subgroup of $H$: we need only to check that $KH$ is a closed subgroup, but $KH=\bigcap_{n\geq 1} KH_n$ is a countable intersection of open (and so closed) subgroups, and so is closed. In this case, if $H$ is moreover cocompact, then $K\cap H$ is a closed cocompact subgroup of $G$.

\vspace{.3cm}  
\begin{lemma}\label{subg}
Let $G$ be locally compact and $\sigma$-compact. Let $H\leq G$ be a(n) closed (resp., open) cocompact subgroup. For each regular $RF$-app. (resp., $RC$-app.) $(G_n)$ of $G$,  $(H\cap G_n)$ is a regular $RC$-app. of $H$. If $H\leq G$ is  an open cocompact subgroup and $(G_n)$  is a regular $RF$-app. of $G$, then $(H\cap G_n)$ is a regular $RF$-app. of $H$.
\end{lemma}
\begin{proof}
By assumption, either $H$ or each $G_n$ is open, or we have both. The result follows from Lemma \ref{intersection} in any case.
\end{proof}

\vspace{.3cm}
\begin{proposition}\label{asdim}
Let $G$ be locally compact and $\sigma$-compact. Given a regular $RF$-app. $\sigma=(G_n)$ of $G$,

$(i)$ ${\rm asdim}(G)\leq{\rm asdim}(\Box_\sigma G)$;

$(ii)$ for any closed cocompact subgroup $H\leq G$ and regular $RF$-app.  $\tau=(H\cap G_n)$ of $H$, ${\rm asdim}(\Box_\tau H)\leq{\rm asdim}(\Box_\sigma G)$.

$(iii)$ for any closed cocompact subgroup $H\leq G$,   ${\rm asdim}(\Box_s H)={\rm asdim}(\Box_s G)$.

\end{proposition}
\begin{proof}
Parts $(i)$, $(ii)$ and the easy inequality in part $(iii)$  follow from Lemmas \ref{basic} and \ref{subg}. We just need to verify that ${\rm asdim}(\Box_s H)\geq{\rm asdim}(\Box_s G)$. Let $d_H$ be a plig metric on $H$. Take a compact set $F_H$ of representatives of $G/H$ such that $F_H\times H\to G; \ (x,h)\mapsto x\cdot h$ is bijective (use  properness of the quotient map and take $F_H:=q^{-1}(G/H)$). Choose $R>0$ large enough that $F_H\subseteq B_R(e)$. Use  Lemma \ref{basic} to find $n\geq 1$ and collections $\mathcal W^{(0)},\cdots, \mathcal W^{(s)}$, with $s={\rm asdim}(\Box_s H)$, forming a uniformly bounded cover of $H$ with Lebesgue number at most $8R$ such that each $\mathcal W^{(j)}$	consists of $(H\cap G_n)$-right translation invariant disjoint sets, for $0\leq j\leq s$. Choose $n$ large enough that $G_n\subseteq H$. For $0\leq j\leq s$, put $\mathcal V^{(j)}=\{B_{-4R}(W): \ W\in \mathcal W^{(j)}\}$, where $B_{-4R}(W):=\{h\in H: \  B_{4R}(h)\cap H\subseteq W\}$, and observe that $\mathcal V^{(j)}$ consists of $G_n$-right translation invariant sets with minimal distance $4R$. By the above bound on the Lebesgue number, any $4R$-ball in $H$ is contained in a member of $\mathcal W^{(j)}$'s, and so $\mathcal V^{(j)}$'s form a cover of $H$. Now $\mathcal U^{(j)}:=\{B_{2R}(e)V: \ V\in \mathcal V^{(j)}\}$ consists of disjoint $G_n$-right translation invariant sets, and the Lebesgue number of the cover $\mathcal U^{(0)}\cup\cdots\cup \mathcal U^{(s)}$ is at most $R$ (for each $g\in G$ there is $x\in F_H$ and $h\in H$ such that $g=xh$ and there is $V\in \mathcal V^{(0)}\cup\cdots\cup \mathcal V^{(s)}$ with $h\in V$. Since $F_H\subseteq B_R(e)$, $g\in B_R(e)V$ and so $B_R(g)=B_R(e)g\subseteq B_{2R}(e)V\in \mathcal U^{(0)}\cup\cdots\cup \mathcal U^{(s)}$, as claimed). This finishes the proof of the reverse inequality. 
\end{proof}

\section{Topologically virtually nilpotent and polycyclic-by-compact groups} \label{tvn}
Recall that a group $G$ is ($r$-step) {\it nilpotent} if there is a finite series
$$1=G_0\unlhd G_1\unlhd\cdots\unlhd G_r=G$$
of normal subgroups with $[G,G_i]\leq G_{i-1}$, for $i=1,\cdots,r$. We say that $G$ is {\it topologically virtually nilpotent} (TVN) if it has a nilpotent cocompact subgroup. Hirsch proved that every finitely generated (discrete) nilpotent group is {\it polycyclic} (i.e., has a finite series with cyclic factors) \cite{h}.  He also showed that polycyclic-by-finite groups are residually finite (RF). We use the following series of lemmas which extend a classical result of Malcev  on polycyclic-by-finite groups \cite{m}, to show the analog of Hirsch result for {\it polycyclic-by-compact} groups, i.e., groups with a polycyclic  cocompact  subgroup. 

For this purpose, we need the notion of Hirsch length for topological groups.
Let $G$ be a topological group and define the {\it Hirsch length} $h(G)$ of $G$ as the maximum number $n$ (possibly infinity) of infinite cyclic factors in all possible normal series
$$1=G_0\unlhd G_1\unlhd\cdots\unlhd G_m=G$$
of closed subgroups of $G$. This is clearly well-defined. When $G$ is discrete, this is the same as  the number of infinite cyclic factors in any series 
$$1=G_0\unlhd G_1\unlhd\cdots\unlhd G_n=G$$
with cyclic or  finite factors, by  the Jordan-H\"{o}lder-Schreier Theorem \cite{b}. Unfortunately this  could not be adapted to the topological case (with finite replaced by compact), as the second isomorphism theorem is not valid for topological groups (except under some extra conditions; c.f., \cite[Theorem 5.3]{hr}). The argument in the algebraic case has two parts: each refinement of a series as above has the same number of infinite cyclic factors as the original series, and that any two such series have isomorphic refinements by the Jordan-H\"{o}lder-Schreier Theorem. It is indeed the second part which fails in the topological case, as the next lemma shows.   

\vspace{.3cm}  
\begin{lemma}\label{hl}
The number of infinite cyclic factors in any series 
$$1=G_0\unlhd G_1\unlhd\cdots\unlhd G_n=G$$
with cyclic or  compact factors remains unchanged after refinement with closed subgroups.
\end{lemma}
\begin{proof}
If we insert normal closed subgroups in the series as
$$G_{i+1}\unlhd A\unlhd B\unlhd G_i$$
to get a proper refinement, then either $G_{i+1}$ is cocompact in $G_i$, in which case  $A/G_{i+1}$ is  compact as a closed subgroup of a compact group, and so are $G_i/B$ and $B/A$, as a quotient or a closed subquotient of $G_i/G_{i+1}$, and nothing is added to the number of non-compact cyclic factors; or $G_{i+1}$ is not cocompact in $G_i$, in which case $G_i/G_{i+1}$ is an infinite cyclic group and has no proper infinite quotient, thus as we have started with a proper refinement, $A$ has to be of finite index in $G_i$ and so is $B$. This forces $G_{i+1}$ to be of infinite index in $A$ and again the number of non-compact cyclic factors is not effected.  This shows that each refinement of the above series has the same number of non-compact cyclic factors. 
\end{proof}

Let us see what goes wrong with  refinements: take another series
$$1=H_0\unlhd H_1\unlhd\cdots\unlhd H_n=G,$$
and observe that, for each $i$,
$$1=G_i\cap H_0\unlhd G_i\cap H_1\unlhd\cdots\unlhd G_i\cap H_n=G_i$$
is an ascending series of closed subgroups. This gives subnormal series
$$G_{i+1}=(G_i\cap H_0)G_{i+1}\unlhd (G_i\cap H_1)G_{i+1}\unlhd\cdots\unlhd (G_i\cap H_n)G_{i+1}=G_i,$$
and repeating this for $i=0,\cdots, m$, we get a refinement of the first series with $mn$ terms. Similarly, we could get a refinement of the second series with $nm$ terms by inserting subnormal series
$$H_{j+1}=(H_j\cap G_0)H_{j+1}\unlhd (H_j\cap G_1)H_{j+1}\unlhd\cdots\unlhd (H_j\cap G_m)H_{j+1}=H_j,$$
for  $j=0,\cdots, n$. These refinements are algebraically isomorphic, since for  subgroups $Q, N,$ and $L$ of $G$ with $L$ normal subgroup of $Q$ satisfying $qN=Nq$, for every $q\in Q$, we have $QN/LN\simeq Q/L(Q\cap N)$ \cite[Lemma 1]{b}). This applies in the first refinement to $Q:=G_i\cap H_j, N := G_{i+l},$ and  $L:=G_i\cap H_{j+1}$, and to the second refinement with interchanging the roles of the
$G_i$'s  and $H_j$'s. However, since one could not guarantee that $QN/LN$ is locally compact, there is no way to show that it is also homeomorphic to $Q/L(Q\cap N)$ (which is clearly locally compact).  

On the other hand, there are situations where Hirsch length is well-behaved (i.e., it satisfies the Hirsch formula). This is the content of  next lemma.  

\vspace{.3cm}  
\begin{lemma}\label{hf}
	If $G$ is a topological group and  $H\unlhd G$ is a closed normal subgroup, then 
	$$h(G)\geq h(H)+h(G/H).$$
	If moreover, $H$ is compact, cocompact, or open, then the equality holds.
\end{lemma}
\begin{proof}
	Let us first assume that $h(H)=n$ and $h(G/H)=m$ are both finite. Then there are normal series
	$$1=H_0\unlhd H_1\unlhd\cdots\unlhd H_k=H$$
	of closed subgroups of $H$ with $n$ infinite cyclic factors and normal series
	$$H/H\unlhd G_1/H\unlhd\cdots\unlhd G_\ell/H=G/H$$
	of closed subgroups of $G/H$ with $m$ infinite cyclic factors. Since the quotient map: $G\to G/H$ is open, each $G_i$ is a closed subgroup of $G$ containing $H$. Now we get   the subnormal series 
	$$1=H_0\unlhd H_1\unlhd\cdots\unlhd H_{k-1}\unlhd H\unlhd G_1\unlhd\cdots\unlhd G_\ell=G$$
	with $n+m$ infinite cyclic factors, thus $h(G)\geq n+m$. 
	
	If one of the terms in  right hand side of the inequality is infinite, say that of $H$, then again, for each $n$, there is a  normal series
	of closed subgroups of $H$ with $n$ infinite cyclic factor, and by the above argument, $h(G)\geq n+m$, for each $n$, thus $h(G)=\infty$, and we again have the inequality.
	
	Next let us assume that $N$ is compact. Then given a subnormal series  $$1=G_0\unlhd G_1\unlhd\cdots\unlhd G_n=G$$
	we get a series 
	$$1\unlhd G_1N/N\unlhd\cdots\unlhd G_{n-1}N/N\unlhd G/N$$
	where each $G_iN$ is a closed subgroup of $G$ (since $G_i$ is closed and $N$ is compact) and so locally compact. It is known that in this case, the second isomorphism theorem holds, that's is, $G_iN/N\simeq G_i/(G_i\cap N)$, as topological groups \cite[Theorem 5.3]{hr}. Also, $G_{i+1}N/G_iN\simeq G_{i+1}/G_i(G_{i+1}\cap N)$. Now there is a group epimorphism: $G_{i+1}/G_i\to G_{i+1}/G_i(G_{i+1}\cap N)$ with kernel $G_{i+1}\cap N$. In particular, if $G_{i+1}/G_i$ is the infinite cyclic group, then $G_{i+1}\cap N$ has to be isomorphic to $k\mathbb Z$, for some integer $k$, but this intersection is a compact group, so $k=0$ and $G_{i+1}N/G_iN$ is the infinite cyclic group. This shows that $h(G/N)\geq h(G)$, and since $h(N)=0$,  equality (Hirsch formula) holds. 
	
	When $N$ is cocompact, we get a series  $$1\unlhd G_1\cap N\unlhd\cdots\unlhd G_{n-1}\cap N\unlhd N$$
	and there is a group monomorphism: $(G_{i+1}\cap N)/(G_i\cap N)\to G_{i+1}/G_i$, so if $G_{i+1}/G_i$ is the infinite cyclic group, $(G_{i+1}\cap N)/(G_i\cap N)$ is either isomorphic to the infinite cyclic group, or is trivial. But if it is trivial, then $G_{i+1}\cap N=G_i\cap N$. In this case, since $N$ is cocompact in $G$, $G_{i+1}\cap N$ is cocompact in $G_{i+1}$, thus in the series of closed subgroups $$G_{i}\cap N\leq G_{i}\leq G_{i+1}$$
	the first is cocompact in the last, and so is the second in the last, i.e., $G_{i+1}/G_i$ has to be compact, which contradicts our assumption. Thus $(G_{i+1}\cap N)/(G_i\cap N)$ is infinite cyclic, whenever  $G_{i+1}/G_i$ is so, that is, $h(N)\geq h(G)$, and since $h(G/N)=0$, again the equality holds.
	
	Finally, if $N$ is open, then $G/N$ is discrete, and given a subnormal series  $$1=G_0\unlhd G_1\unlhd\cdots\unlhd G_n=G$$
	in the series 
	$$1\unlhd G_1N/N\unlhd\cdots\unlhd G_{n-1}N/N\unlhd G/N$$
	all factors $G_iN/N$ is discrete and $G_iN/N\simeq G_i/(G_i\cap N)$. Also $G_iN$ is open and $G_{i+1}N/G_iN\simeq G_{i+1}/G_i(G_{i+1}\cap N)$, with both sides discrete. Again there is a group epimorphism: $G_{i+1}/G_i\to G_{i+1}/G_i(G_{i+1}\cap N)$ with kernel $G_{i+1}\cap N$. If $G_{i+1}/G_i$ is the infinite cyclic group, then $G_{i+1}\cap N$ has to be isomorphic to $k\mathbb Z$, for some integer $k$. But then $G_i\cap N$ also has to be of the form $\ell\mathbb Z$, for some integer $\ell$ (deviding $k$). Since there is a group monomorphism: $(G_{i+1}\cap N)/(G_i\cap N)\to G_{i+1}/G_i$, and $G_{i+1}/G_i$ has no non-trivial finite subgroup, both $k$ and $\ell$ cannot be non-zero. If $k=0$, then $\ell=0$ and we get no infinite factor in the subnormal series of $N$ at the $i^{\rm th}$ position, but get one such factor in the subnormal series of $G/N$ at the $i^{\rm th}$ position. If $k\neq 0$, then  $\ell=0$ and we get one infinite factor in the subnormal series of $N$ at the $i^{\rm th}$ position. This shows that  
	$$h(N)+h(G/N)\geq h(G),$$
	 finishing the proof.	 
\end{proof}

Now since the polycyclic groups have finite Hirsch length we immediately get the following result from the cocompact case of the above lemma. 

\vspace{.3cm}  
\begin{corollary}\label{pc}
	Polycyclic-by-compact groups have finite Hirsch length.
\end{corollary}

\vspace{.3cm}  
One could extend the notion of polycyclic groups by considering a family $\mathfrak X$ of locally compact groups and call $G$ a {\it poly-$\mathfrak X$ group} if $G$ has a finite subnormal series
$$1=G_0\unlhd G_1\unlhd\cdots\unlhd G_n=G$$
whose factors belong to $\mathfrak X$. The next lemma extends and is proved similar to \cite[10.2.4]{p}.

\vspace{.3cm}  
\begin{lemma}  \label{polyX}
If $\mathfrak X$ a family  of locally compact groups, stable under taking closed subgroups (resp. under taking quotients by closed normal subgroups), then so is the class of poly-$\mathfrak X$ groups.
\end{lemma}

\vspace{.3cm}  
\begin{lemma}  \label{polycycl}
Each poly-$\{$cyclic, compact$\}$ group has a
characteristic closed cocompact subgroup that is poly-$\{$infinite cyclic$\}$.
\end{lemma}
\begin{proof}
We adapt the proof of \cite[10.2.5]{p}. Consider the situation that $N\unlhd L$ is compact with $L/N$ infinite cyclic. If $L = \langle N, x\rangle$, for some $x\in L$, then
some power $x^n$ of $x$ with $n\geq 1$ centralizes both $N$ and $x$. Therefore,
$x^n$ is central in $L$, and in particular, $\langle x^n\rangle$ is a normal infinite cyclic
cocompact closed subgroup of $L$.

Now take a poly-$\{$cyclic, compact$\}$ group $G$ and take a finite subnormal series
$$1=G_0\unlhd G_1\unlhd\cdots\unlhd G_n=G$$
with cyclic or compact quotients. We argue by induction on $i$ to show that each $G_i$ has a characteristic closed cocompact subgroup $H_i$ that is poly-$\{$infinite cyclic$\}$. This is trivial for $i=0$, and if it holds for $i$, then since $G_i\unlhd G_{i+1}$ and $H_i$ is characteristic
in $G_i$, then $H_i \unlhd G_{i+1}$. Let us first observe that $G_{i+1}$ has a normal closed cocompact poly-$\{$infinite cyclic$\}$ subgroup. This is clear if $G_{i+1}/G_i$ is compact, thus we may assume that $G_{i+1}/G_i$ is infinite cyclic. In this case, $L := G_{i+1}/H_i$ has a
compact normal subgroup $N = G_{i}/H_i$ with $L/N= G_{i+1}/G_i$  infinite cyclic. By the observation of the previous paragraph,  $L$ has a normal infinite cyclic
cocompact closed subgroup. The inverse image $M$ in $G_{i+1}$ of this group under the quotient map is then a normal poly-$\{$infinite cyclic$\}$ closed cocompact subgroup of $G_{i+1}$. In particular, $M$ is finitely generated and so countable. Since $M$ is locally compact in the relative topology of $G$, it has to be discrete. By the classical Poincare Lemma, $M$ contains a closed cocompact characteristic subgroup of $G_{i+1}$ (c.f., \cite[Theorem 7.1.7]{s}). Let us take $H_{i+1}$ to be this subgroup, which is  poly-$\{$infinite cyclic$\}$ by Lemma \ref{polyX}. This finishes the argument of the inductive step, and the induction stops at $G = G_n$. 
\end{proof}

\begin{lemma}  \label{tf}
Each noncompact  polycyclic-by-compact group has a closed nontrivial discrete finitely generated 
  torsion-free abelian characteristic subgroup.
\end{lemma}
\begin{proof}
Let $G$ be a noncompact  polycyclic-by-compact group. By Lemma \ref{polycycl}, $G$ has a  closed cocompact poly-$\{$infinite cyclic$\}$ characteristic subgroup $H$. Since $G$ is not compact, $H\neq 1$. By definition, $H$ is discrete, finitely generated and solvable. Take  the
derived series
$$H=H_0\unrhd H_1\unrhd\cdots\unrhd H_{n+1}=1$$
for $H$, where $n$  is minimal with $H_{n+1}=1$, then $H_n$ is
a nontrivial abelian characteristic  subgroup of $G$. Finally, $H_n\leq H$
 is torsion free and hence infinite.
\end{proof}

We have already used an inductive argument (based on the indices in subnormal series) in the proof of Lemma \ref{polycycl}. Another instance of inductive argument, this time based on the Hirsch length, is used in the proof of the following extension of a classical result of Malcev for polycyclic-by-finite groups \cite{m}.

\vspace{.3cm}  
\begin{lemma} [Malcev] \label{pbc}
If $G$ is polycyclic-by-compact, then each closed subgroup $H\leq G$ is the intersection of all closed cocompact subgroups of $G$ containing $H$.
\end{lemma}
\begin{proof}
Let $\tilde H$ be the intersection of all closed cocompact subgroups of $G$ containing $H$. To show that $K=H$, we proceed by induction on the Hirsch length $h(G)$. The result is trivial when $G$ is compact, that is, $h(G)=0$. When $G$ is not compact, it has a non trivial discrete finitely generated torsion-free abelian normal subgroup $A$. Since $A$ is abelian and torsion-free, the map $x\mapsto x^r$ is a group monomorphism. Let us denote its range by $A^r$. Then $A^r$ is a discrete (and so closed) normal subgroup of $G$. Also, since $A$ is finitely generated and torsion-free, it is free abelian, that is, $A$ is isomorphic to $\mathbb Z^n$, for some $n$ (which is then non-zero, as $A$ is not trivial). In particular, $\bigcap_{r\geq 1} A^r=1$ and $h(A)>0$. But then, since $A$ and $A^r$ are isomorphic, $h(A^r)>0$. It follows from Lemma \ref{hf} that $h(G/A^r)<h(G)$.  By the induction hypothesis applied to $G/A^r$ we have,
$$H\leq \tilde H\leq \bigcap_{r\geq 1} HA^r \leq HA,$$
hence
$$\tilde H\leq H(A\cap \bigcap_{r\geq 1} HA^r)=H\big(\bigcap_{r\geq 1} (H\cap A)A^r)\big)=H(H\cap A)=H,$$
which finishes the proof.
\end{proof}

\begin{lemma}  \label{sgp}
Let $G$ be polycyclic-by-compact and $H$ be a closed subgroup. Then $h(H)=h(G)$ if $H$ is cocompact. Conversely, if $H$ a closed normal subgroup with $h(H)=h(G)$, then  $H$ is cocompact.
\end{lemma}
\begin{proof}
First assume that $H$ is cocompact and $N$ be the cocompact closed normal  subgroup of $G$ contained in $H$, constructed in the proof of Poincar\'{e} lemma. Then clearly $h(G/N)=h(H/N)=0$, and so $h(H)=h(N)=h(G)$, by the Hirsch formula in the cocompact case.
Conversely, if $H$ is normal, then $$h(G)\geq h(H)+h(G/H)=h(G)+h(G/H),$$
thus $h(G/H)=0$. Since $G/H$ is also polycyclic-by-compact, it could only have zero Hirsch length if it is compact. Therefore, $H$ is cocompact.
\end{proof}

Note that in the discrete case, the above result  could be proved without normality assumption by an inductive argument on the Hirsch length \cite[10.2.10]{p}. Let us see what goes wrong in the topological case. To use induction on $h(G)$, let us first consider the case $h(G)=0$, then $G$ is compact and the claim is trivial. If $G$ is noncompact, one could use Lemma \ref{tf} to choose a normal torsion-free infinite discrete abelian subgroup $A$
of $G$. If we knew that the Hirsch formula
$$h(H) = h(H \cap A) + h(H/H \cap A)$$
holds, the rest would follow: the right hand side of the above equality is less than or equal 
$h(A) + h(G/A)$ which in turn is at most $h(G),$
and since $h(H) = h(G)$, it follows that,
$h(H \cap A) = h(A)$, which yields $h(A/H \cap A) = 0$, showing that $H \cap A\leq A$ is cocompact. Now choose a closed locally cocompact characteristic subgroup $B$ of $A$ contained in $H \cap A$ by Lemma \ref{char}. Since $A\unlhd G$ and $B\leq A$ is characteristic, we get $B\unlhd G$. Let us observe that $h(B)>0$, indeed, if $h(B)=0$, then $B$ is compact, but then both $B$ and $A/B$ are locally elliptic, and so is $A$ \cite[4.D.6]{ch}. In particular, $A$ has to be locally finite by \cite[4.D.2]{ch}, which in turn contradicts \cite[2.E.17(3a)]{ch}, as $A$ has no non-trivial finite subgroup (since it is torsion-free). This proves the claim,  which implies that   $h(H/B) = h(G/B) < h(G)$. Now by the induction hypothesis, we have
$H/B\leq G/B$ is cocompact and so is $H\leq G$. 

\vspace{.3cm}  
\begin{proposition} [Hirsch] \label{rc}
Polycyclic-by-compact groups are residually compact ($RC$) and always have a regular $RC$-approximation.
\end{proposition}
\begin{proof}
Take any cocompact closed subgroup $L\leq G$ and take $N:=\bigcap_{g\in G} L^g$ which is the largest closed normal subgroup of $G$ contained in $L$. Choose a compact subset $F_L\subseteq G$ of the representatives of the transversal decomposition $G=\bigcup_{x\in F_L} Lx$ into right $L$-cosets and observe that $N=\bigcap_{x\in F_L} L^x$ is cocompact. Take any closed subgroup $H\leq G$, and observe that in the notations of the Malcev lemma,
\begin{align*}
\tilde H &:=\bigcap\{L: H\leq L\leq G, L \ \text{is closed and cocompact}\}\\
&:=\bigcap\{HN: N\unlhd G, N \ \text{is closed and cocompact}\}.
\end{align*}
Applying this to $H=1$ we get $\tilde H=1$, i.e.,
$$\bigcap\{N: N\unlhd G, N \ \text{is closed and cocompact}\}=1,$$
that is, $G$ is $RC$, with a regular $RC$-approximation.
\end{proof}

\begin{remark}\label{rem}
Lemma \ref{sgp} could be used to show that polycyclic-by-compact groups are indeed either compact or residually finite ($RF$):  we argue by induction on the Hirsch length $h(G)$. Assume that
$G$ is noncompact. Let $A$ be a normal torsion-free infinite abelian subgroup of $G$
given by Lemma \ref{tf}. Then the set $A_r$ consisting of $r$-th powers of elements of $A$  is a  characteristic subgroup of $A$ with finite index and $\bigcap_{r\geq 1} A_r = 1$. Since $G/A_r$ is finite (and so discrete), $A_r$ is an open subgroup.
On the other hand, $h(G/A_r) < h(G)$, for each $r$, and by the inductive assumption, $G/A_r$ is $RF$. Now since
$\bigcap_{r\geq 1} A_r = 1$, it follows that $G$ is also $RF$.

\end{remark}

\vspace{.3cm}
Next we want to show that compactly generated topologically virtually nilpotent groups are residually compact (indeed, either compact or algebraically residually compact). We first need a modification of result due to Hirsch, who showed that finitely generated nilpotent groups are polycyclic \cite{h}. 

\vspace{.3cm}
\begin{lemma} [Hirsch] \label{pcbc}
	Compactly generated nilpotent groups are polycyclic-by-compact.
\end{lemma}
\begin{proof}
	We adapt the inductive argument of Hirsch. Let $G$ be nilpotent of class $r$ with a compact set $S$ of generators. If $r=1$ then $G$ is a compactly generated abelian group, which is then polycyclic-by-compact by \cite[Proposition 33]{mo}. If $r>1$, then let $G_1$ be the last  non-trivial term in the upper central series of $G$. Without loss of generality we may assume that $G_1$ is closed (otherwise use its closure instead) and so compactly generated by \cite[Proposition 5.A.7]{ch}. Then $G_1$ is abelian and compactly generated and so polycyclic-by-compact, by the argument in case of $r=1$. Let $A\leq G_1$ be a polycyclic cocompact subgroup. By the argument of \cite[Proposition 33]{mo}, we may also take $A$ to be closed. Use the same argument as in the proof of Lemma \ref{polycycl} to choose a cocompact closed characteristic subgroup $B$ of $G_1$ contained in $A$. Then since $A$ is polycyclic, so is $B$. Choose a compact subset $K\subseteq G_1$ with $G_1=KB$. Since $G/G_1$ is compactly generated and nilpotent of class $r-1$, by induction hypothesis  it is polycyclic-by-compact. Choose a polycyclic cocompact  subgroup $N/G_1$ and compact subset $E\subseteq G/G_1$ with $G/G_1=E(N/G_1)$. By \cite[Lemma 2.C.9]{ch}, we may choose a compact subset $L\subseteq G$ with $LN/G_1=E$. Thus $G=LN$. Finally, since $N/G_1$ is polycyclic, it is finitely generated, hence there is a finitely generated subgroup $C$ of $N$ with $N=CG_1$. But $C$ is also nilpotent as a subgroup of $G$, thus it is polycyclic-by-finite \cite{h}, that is $C=FD$, for a polycyclic group $D$ and a finite subset $F$. Now we have 
	$$G=LN=CLG_1=LFDG_1=LFG_1D=LFKBD,$$
	where the fourth equality follows from the fact that $G_1$ is normal in $N$. Since $B\leq G_1$ is a characteristic subgroup and $G_1\unlhd G$, we have $D\unlhd G$, and so $BD$ is a subgroup of $G$. But $BD/D\simeq B/(B\cap D)$ algebraically, and the right hand side is polycyclic as a quotient of a polycyclic group. Thus $BD$ is polycyclic as an extension of a polycyclic group by another polycyclic group \cite[7.1.13(c)]{s}. On the other hand, by the continuity of the product map, $LFK$ is compact in $G$, and so $BD$ is a cocompact subgroup of $G$. Therefore, $G$ is polycyclic-by-compact and the inductive argument is complete.
\end{proof}

\begin{lemma}  \label{tfree}
	Each compactly generated topologically virtually nilpotent group contains a finitely generated torsion-free nilpotent cocompact subgroup.
\end{lemma}
\begin{proof}
	Let $G$ be compactly generated and TVN. By definition, $G$ has a closed nilpotent cocompact subgroup $H$. Then $H$ is compactly generated \cite[2.C.8(3)]{ch} and so polycyclic-by-compact by Lemma \ref{pcbc}. Choose a polycyclic cocompact subgroup $A$ of $H$ and observe that $A$ is torsion-free-by-finite \cite[Corollary 2.7]{w}. The finite index torsion-free subgroup $B$ of $A$ would then be cocompact in $G$, finitely generated (as it is a finite index subgroup of the finitely generated group $A$), and nilpotent (as it is a subgroup of the nilpotent group $H$).
\end{proof}

Note that in the above proposition we could not guarantee that $G$ has a cocompact {\it closed} polycyclic subgroup, since though the closure of a cocompact subgroup is again cocompact, the closure of a polycyclic subgroup may fail to be polycyclic (e.g. take the dense polycyclic subgroup $\mathbb Z+\alpha\mathbb Z$ of $\mathbb R$, for $\alpha$ irrational). 

Lemma \ref{pcbc} and Proposition \ref{rc} prove the following extension of a classical result due to Hirsch (c.f. \cite[2.10, 2.13]{w}) for topological groups.

\vspace{.3cm}
\begin{proposition} [Hirsch] \label{rc2}
	Topologically virtually nilpotent (TVN) compactly generated (resp., noncompact) groups are  residually compact (resp., residually finite) and  have a regular $RC$-app. (resp., $RF$-app.). 
\end{proposition}

\vspace{.3cm}
Combining Remark \ref{rem} and Propositions \ref{asdim}, \ref{rc}, and \ref{rc2}, we get the first main result of this section. 

\vspace{.3cm}
\begin{theorem}  \label{main}
	Topologically virtually nilpotent groups and polycyclic-by-compact groups have standard box space.
\end{theorem}

\vspace{.3cm}
Finally, using Lemma \ref{tfree}, we get the second main result of the section.

\vspace{.3cm}
\begin{theorem}  \label{main2}
	Topologically virtually nilpotent compactly generated groups have finite asymptotic dimension.
\end{theorem}
\begin{proof}
	Let $G$ be compactly generated and TVN. Choose a finitely generated  torsion-free nilpotent cocompact subgroup $A$ in $G$ by Lemma \ref{tfree}. By \cite[Theorem 5.2]{j}, there is an embedding of $A$ as a subgroup into the unitary group $U_n(\mathbb Z)$, for some $n\geq 2$. But asdim$\big(\Box_s U_n(\mathbb Z)\big)=n(n-1)/2$, for $n\geq 2$ \cite[Theorem 4.9]{swz}. Hence $A$ has finite asymptotic dimension. Now, since $A$ is a cocompact subgroup, $G$ has finite asymptotic dimension by Proposition \ref{asdim}(iii). 
\end{proof}

Note that in Theorems \ref{main}, \ref{main2}, one cannot guarantee the existence of an {\it open} nilpotent  cocompact subgroup (such a cocompact subgroup exists but we only know that it is closed). Thus in order to use Proposition \ref{asdim}, we need to use the fact that these groups are either compact or algebraically residually compact.  

In Section \ref{ea} we prove more general results of this nature (Theorem \ref{main3}). For that, we would need the following result.

\vspace{.3cm}	
\begin{proposition}  \label{va}
If $G$ is a compactly generated, abelian-by-compact, locally compact group then asdim$(G)=h(G)$.

\end{proposition}
\begin{proof}
	First assume that $G$ is abelian. By \cite[Proposition 33]{mo}, there is a non negative integer $n$, and a cocompact subgroup $A$ of $G$, topologically isomorphic to $\mathbb Z^n$. Since $A$ is cocompact in $G$, we have asdim$(G)$=asdim$(A)=n$, by Lemma \ref{exact}. On the other hand, by definition, $h(G/A)=0$, and so $h(G)=h(A)=n$, by the Hirsch formula in the cocompact case.
	Next consider the case where $G$ is abelian-by-compact. Then exactly by the same argument as above, $G$ has the same asymptotic dimension and Hirsch length as its cocompact abelian subgroup (which could clearly  assumed to be closed and so locally compact) and the result follow from the previous case.	
\end{proof}

This plus Lemma \ref{exact} (proved in Section \ref{ea}) shows the following more general result. 

\vspace{.3cm}
\begin{corollary}  \label{va2}
	If $G$ is a compactly generated, solvable-by-compact, locally compact group then ${\rm asdim}(G)\leq h(G)$.	
\end{corollary}

\vspace{.3cm}
We  have the following sharp estimate for case of polycyclic-by-compact groups.

\vspace{.3cm}
 \begin{proposition}  \label{pcc}
 	If $G$ is polycyclic-by-compact,  ${\rm asdim}(G)= h(G)$.	
 \end{proposition}
\begin{proof}
	Both the asymptotic dimension and  Hirsch length are
	stable under passing to a cocompact subgroup, so we may assume that $G$ is polycyclic, but the result in this case is already known \cite[Theorem 3.5]{ds}.	
\end{proof}

\section{Elementary amenable groups} \label{ea}

A discrete group is elementary amenable (EA) if it can be constructed from finite and abelian groups through the processes of direct unions and extensions. The class was first considered by von Neumann in the study of the Banach-Tarski paradox \cite{v}. Although the class [EA] already contains non discrete groups (including all locally compact abelian ones), it is quite natural to extend the notion by adding some more topological flavor. We could not trace the following natural topological version in the literature: the class of {\it topologically elementary amenable} (TEA) groups is the smallest family of locally compact  Hausdorff groups 
containing all compact and all locally compact abelian Hausdorff groups, and is closed under 
taking closed subgroups, quotients by closed normal subgroups, topological group extensions and increasing unions. Note that starting with a locally compact Hausdorff group and a closed (and normal if needed) subgroup, we end up a locally compact Hausdorff group in the first three cases, and the same holds in the fourth one if we take care of topology as follows: given an increasing family of locally compact Hausdorff groups indexed by a net such that each group sits in a group with larger index as an open subgroup, topologize the union by letting a subset of the union to be open iff its intersection with each of the groups in the family is open. This way we get a locally compact Hausdorff group again. 

Let $\mathcal C, \mathcal D$ be classes of locally compact groups (which are also assumed to be Hausdorff from now on, without further notice). Then the classes $L\mathcal C$ and $\mathcal C\mathcal D$ are the class of ``locally $\mathcal C$'' groups (that is, locally compact groups with all compactly generated locally compact subgroups in $\mathcal C$) and that of  $\mathcal C$-by-$\mathcal D$ groups (that is, groups with a compact or cocompact normal subgroup in  $\mathcal C$ whose quotient group is in  $\mathcal D$). Then the class [TEA] could be constructed by transfinite induction starting with $\mathfrak X_0=\{1\}$, $\mathfrak X_1$ be the class of abelian-by-compact groups, and $\mathfrak X_{\alpha+1}:=(L\mathfrak X_\alpha)\mathfrak X_1$, and for a limit ordinal $\beta$, $\mathfrak X_\beta:=\cup_{\alpha<\beta}\mathfrak X_\alpha$. Then adapting an argument of  Hillman and Linnel in \cite{hl}, we have [TEA] is nothing but the union of all $\mathfrak X_\alpha$'s. Now, following the idea of Hillman and Linnel, we define the Hirsch length on the class  $\mathfrak X_1$  of abelian-by-compact groups as in the previous section\footnote[1]{We do not follow the footsteps of \cite{fw} here, where they define the Hirsch length of an abelian-by-finite group $A$ as the module dimension dim$_{\mathbb Q}(A\otimes \mathbb Q)$, as in \cite{ds}. This causes no problem as long as $A$ is countable, but even for an innocent looking group like $\mathbb R$, the module dimension is $\infty$, where as the Hirsch length--based on our definition in previous section--is 1 (c.f. \cite[Remark 3]{ds}). The same of hesitation is needed when one defines different notions of asymptotic dimension. For instance the asymptotic dimension of $\mathbb R$ with respect to the coarse structure consisting of subsets $E\subseteq \mathbb R\times \mathbb R$ with $\{x-y: (x,y)\in E\}$ finite, is $\infty$ \cite[Remark 2]{ds}, where as it is 1, if we replace ``finite'' by ``relatively compact'' in the above coarse structure. The main idea of this paper is that that such pathologies disappears if one take the topology into account when working with classes such as solvable-by-compact groups.}, on groups in $G\in L\mathfrak X_\alpha$ by $h(G):=\sup\{h(H): H\in \mathfrak X_\alpha \}$, and on groups $G\in\mathfrak X_{\alpha+1}$ having a compact or cocompact normal subgroup $N\in \mathfrak X_{\alpha}$ with $G/N\in\mathfrak X_{1}$, by $h(G)=h(N)+h(G/N)$. Again the inductive argument of \cite[Page 238]{hl} could be adapted to show that the Hirsch length is well-defined for  $G$ in [TEA], and satisfies the natural properties: $h(H)\leq h(G)$, for each closed subgroup $H$, $h(G)=h(N)+h(G/N)$, for each compact or cocompact normal subgroup $N$, and $h(G)=\sup\{h(H): H\leq G \ {\rm closed \ and \ compactly \ generated} \}$ (c.f. \cite[Theorem 1]{h}). This plus the Hirsch formula in the previous section show that the notion defined here coincides with the notion defined in previous section for polycyclic-by-compact groups.  

The class [EA] is known to include all virtually solvable groups (and so all virtually nilpotent and all polycyclic-by-finite groups), but as we assume that [TEA] is stable only under taking ``closed’’ subgroups, we would not have all topologically virtually nilpotent, or even all polycyclic-by-compact groups in the class [TEA]. The choice of assuming stability under taking arbitrary subgroups has its own price to pay: the class [TEA] would then include the class of maximally almost periodic groups, and in particular the non amenable free groups would be in [TEA] (c.f., \cite[Example 12]{pa}).  

Let us compare and contrast the class [TEA] with other known classes of topological groups. Our reference for the properties of various classes discussed here is \cite{pa}, from which we borrow class notations as well. [TEA] clearly contains the classes [Z] of {\it central} groups (groups whose quotient over center is compact) which is included in the larger class of extensions of abelian groups by compact groups. It does not contains the class [MAP] of {\it maximally almost periodic} groups (groups with enough finite dimensional irreducible representations to separate the points, or equivalently, enough continuous almost periodic functions to separate the points), as these are subgroup of compact groups and so contains many non amenable groups (this justifies our assumption that the class [TEA] is stable only under taking closed subgroups). [MAP] includes the class [Moore] of Moore groups (groups whose irreducible representations are finite dimensional) by Gelfand-Raikov theorem. Note that a Lie group belongs to [Moore] if and only if it is a finite extension of a central group, so [TEA] contains Lie groups in the class [Moore]. For compactly generated groups, [MAP] is included in [SIN] ({\it small invariant neighborhood} groups, so the latter is not included in [TEA], indeed [SIN] includes all discrete groups (more generally, [SIN]-groups are discrete extensions of direct products of vector groups and compact groups (c.f., \cite[Theorem 2.13]{gm}). However connected [SIN] groups are in [TEA] (as they are in [Z], by Freudenthal-Weil theorem), though there are connected groups in [TEA] included neither in [MAP] nor in [SIN] (c.f., \cite[Example 8]{pa}) as well as compactly generated totally disconnected groups in [TEA] included neither in [MAP] nor in [SIN] (c.f., \cite[Example 9]{pa}). More generally, connected [IN]-groups are extension of compact groups by vector groups, and so are in [TEA] \cite[Theorem 2.9, Corollary 2.8]{gm}. A less trivial inclusion follows by a result of Grosser and Moskowitz \cite[Proposition 4.5]{gm}, stating that a compactly generated [FIA]$^{-}$ group ({\it topologically finite inner automorphism group}, that is a group with relatively compact inner automorphisms group) is an extension of a direct product of some vector group $\mathbb R^n$ with a compact group by some $\mathbb Z^d$, showing that [TEA] contains compactly generated [FIA]$^{-}$ groups. Since [SIN]$\cap$[FD]$^{-}\subseteq$ [FIA]$^{-}$ \cite[Theorem 4.6]{gm}, same is true for compactly generated [SIN]-groups which are in [FD]$^{-}$ ({\it topologically finite derived subgroup}, that is a groups with relatively compact commutator or derived subgroup). This class is quite interesting as it contains non Type I groups, showing that [TEA] is not included in [Type I]: there is a central extension of $\mathbb T$ by $\mathbb Z^2$ with derived subgroup $\mathbb T$, which is a compactly generated group in [SIN]$\cap$[FD]$^{-}$, with a non Type I, $\mathbb Z^2$ multiplier (and so non Type I) (c.f., \cite[Example 7]{pa}; also there are compactly generated solvable non Type I groups, c.f., \cite[Example 17]{pa}; and connected solvable non Type I groups, c.f., \cite[Example 19]{pa}). Another structural result of Grosser and Moskowitz \cite[Theorem 3.16]{gm}, states that for a group in  [FC]$^{-}$ ({\it topologically finite conjugacy class} groups, that is groups relatively compact conjugacy classes) the quotient by the intersection of all compact normal subgroups is a direct product of a
vector group and a discrete torsion free abelian group. This shows that 
[FC]$^{-}$ is included in [TEA] (note that for compactly
generated groups, [FC]$^{-}$= [FD]$^{-}$ (c.f., \cite[Theorem 3.20]{gm}), however there is a semidirect product of $\mathbb R$ by $\mathbb Z^2$ (which is then in [TEA]) not included in [FC]$^{-}$ (c.f., \cite[Example 11]{pa}).

In this section we adapt the argument in \cite{fw} to show that for  second countable topologically elementary amenable groups, the asymptotic dimension of any box space is bounded above by the Hirsch length. Since we need a plig metric (see section \ref{Box}) in order to make sense of asymptotic dimension, one first needs to restrict   to second countable locally compact groups. Such a group is then automatically $\sigma$-compact, and so compactly generated by \cite[2.C.6]{ch}. For the general case, we define asdim$(G)$ as the supremum of asdim$(K)$, over all compactly generated  subgroups  $K\leq G$ (or equivalently, over all compactly generated closed  subgroups by \cite[2.C.2(5)]{ch}). This is then the same as the asymptotic dimension of $G$ as a topological group endowed with its canonical left coarse structure \cite[Theorem 2.8]{bcl}.

A locally compact group $G$ is called {\it locally elliptic} if every compact subset of $G$ is contained in a compact open subgroup. These are exactly those locally compact groups with zero asymptotic dimension \cite[4.D.4]{ch}.  

\vspace{.3cm}
\begin{lemma}  \label{exact}
Let $G$ be a locally compact, compactly generated group. For a short exact sequence $$1\rightarrow N\rightarrow G\rightarrow G/N\rightarrow 1$$ with $N$ compact or cocompact, we have ${\rm asdim}(G)\leq {\rm asdim}(N)+{\rm asdim}(G/N)$, and equality holds if the exact sequence is split and $N$ is locally elliptic. 
\end{lemma}
\begin{proof}
Since $G$ is compactly generated,  so are $N$ and $G/N$ by \cite[2.C.8]{ch}. Take any plig metric on $G$ (no matter which one, as they are all coarsely equivalent) and observe that the quotient map $q: G\to G/N$ is an isometry (and so Lipschitz) with respect to the canonical quotient metric on $G/N$, and that the coarse fibres are translates of a neighborhood of $N$ (and so coarsely equivalent to $N$). Since $G$ is coarsely equivalent to a geodesic metric space \cite[4.B.9]{ch}, the desired inequality  follows from \cite[Theorem 29]{bd}. 
If $N$ is  locally elliptic, then  ${\rm asdim}(N)=0$ and ${\rm asdim}(G)\leq {\rm asdim}(G/N)$. On the other hand, since the exact sequence is split (in the category of topological groups), $G/N$ is isomorphic (as a topological group) to a closed subgroup of $G$, and so ${\rm asdim}(G)\geq {\rm asdim}(G/N)$, and we get the equality.  	
\end{proof}

The first main result of this section reads as follows.

\vspace{.3cm}
\begin{theorem}  \label{main3}
	Let $G$ be topologically elementary amenable (TEA), then  asdim$(G)\leq h(G)$. In particular groups in {\rm [}TEA{\rm ]} with finite Hirsch length have finite asymptotic dimension.
\end{theorem}
\begin{proof}
Assume first that $G$ is compactly generated. By Proposition \ref{va} we have the the desired inequality  (even equality) when $G$ is in class $\mathfrak X_1$. By lemma \ref{hf},
the Hirsch length satisfies $h(G) = h(N) + h(G/N)$, for any compact or cocompact normal subgroup $N$, whereas the asymptotic dimension satisfies  asdim$(G)\leq$ asdim$(N) +$ asdim $(G/N)$, by Lemma \ref{exact}. The
result now follows by a transfinite induction on the ordinal $\alpha$ where $G$ belongs to $\mathfrak X_\alpha$. 
In general, we could calculate both the asymptotic dimension and Hirsch length by taking supremum of the same invariant over all compactly generated closed subgroups. Since being (TEA) passes to closed subgroups, the result follows from the previous case.
\end{proof}

The second statement in the above theorem has a more structural version (Corollary \ref{sbc}), which extends \cite[Corollary 1]{hl}. First we need the following adaptation of \cite[Main Theorem]{hl}.  We use the notation  $\Lambda(G)$ to denote the maximal locally elliptic closed normal subgroup of $G$, which exist by Zorn lemma (but might be trivial), is unique \cite[4.D.7(7)]{ch}, and is called the {\it locally elliptic radical} of $G$. Recall that an action of $G$ on a topological space is called {\it effective} (or faithful) if for every element $g \neq e$ in $G$, there exists $x \in X$ with $g\cdot x \neq x$. 

 \vspace{.3cm}
\begin{lemma}  \label{scc}
	For each $G$ in {\rm [}TEA{\rm ]} of finite Hirsch length, $G/\Lambda(G)$ has a maximal solvable normal closed cocompact subgroup.
\end{lemma}
\begin{proof}
	Assume first that $G$ is compactly generated. Arguing by induction, let the result be true for  such groups with
	Hirsch length at most $n$ and let $G$ be in [$TEA$] of Hirsch length $n+1$. Then by transfinite induction, one could show that $G$ has a normal closed subgroup $K$ with $G/K$ infinite and abelian-by-compact. Modding out the maximal compact normal subgroup, we may assume that $G/K$ is abelian, but not compact. Since $G/K$ is compactly generated \cite[2.C.8(4)]{ch}, it is $\mathbb Z^{r}$-by-compact, for some $r>0$ \cite[Proposition 30]{mo}. In particular, $h(G/K)=r>0$. Choose a subgroup $H$ of $G$  containing $K$ such that $H/K$ 	is a maximal abelian normal closed subgroup of $G/K$, then $H$ is cocompact in $G$. Then the action of $G/H$ on $H/K$ (by conjugation) is effective, and $r = h(G/K) < h(G) = n+1$. Since $ h(K) \leq h(G) -r<n$, by the
	inductive hypothesis, $K$ has a normal cocompact closed subgroup $L$ containing
	$\Lambda(K)$ with $L/\Lambda(K)$ a  maximal solvable normal subgroup of $K/\Lambda(K)$ say with derived length $d$. As both $\Lambda(K)$ and $L$ are characteristic in $K$, they are also normal in $G$. Moreover, $\Lambda(K)=K\cap \Lambda(G)$. Now the centralizer of $K/L$ in $H/L$ is a normal solvable closed cocompact subgroup of $G/L$ with derived length at most 2. Finally, $G/A(G)$ has a maximal solvable normal closed cocompact subgroup of derived length at most $2+d$, since it contains the pre-image of the centralizer of $K/L$ in $H/L$. This completes the inductive argument in the cocompact case. 
									
	In general, let $\{G_i\}_{i \in I}$ be the set of compactly generated subgroups of $G$. By the above argument, each $G_i$ has a normal closed subgroup $H_i$ containing $\Lambda(G_i)$ such that $H_i/\Lambda(G_i)$ is the maximal	solvable normal closed cocompact subgroup of $G_i/\Lambda(G_i)$ with derived length at most $2+d$, for $d$ as in the previous case. Clearly $H_i$ is the maximal closed normal locally elliptic-by-solvable subgroup of $G_i$.  For any index $j\in I$, $H_i \cap G_j$ is a normal locally elliptic-by-solvable subgroup of $G_j$, and so $H_i \cap G_j\leq H_j$. Also if $G_j\leq G_i$, then $H_j\leq H_i$. Let $H$ be the union of all $H_i$'s and for $ x, y \in  H$ and $g \in G$, choose indices $i, j, k \in I$ with $x\in	H_i, y \in H_j$, and $g \in G_k$. Choose $\ell\in J$ such that $G_\ell$ contains
	$G_i \cup G_j \cup G_k$. Then $x, y \in H_\ell$ (by the above argument) and so are  $xy^{- l}$ and $gxg^{-1}$, that is, $H$ is a normal  (not necessarily closed) cocompact subgroup of $G$. 
    Let $D_i$ be the $2+d$ derived subgroup of $H_i$. Then $D_i$ is a locally elliptic normal subgroup of $G_i$, and  by an argument similar to the above, $D:=\bigcup_{i\in I} D_i$ is a locally elliptic (not necessarily closed) normal subgroup of $G$. But the $2+d$ derived subgroup of $H$ is contained in  $D$ (as each iterated commutator involves only finitely many elements of $H$), thus by the algebraic isomorphism, $H\Lambda(G)/\Lambda(G)\simeq H/(H \cap \Lambda(G))$, the left hand side is solvable of derived length at most $2+d$, and the inductive argument is complete in the general case. 
\end{proof}

As a result we get the following corollary which extends \cite[Corollary 1]{hl} and \cite[Theorem 2]{h}. 

\vspace{.3cm}
\begin{corollary}  \label{sbc}
	A topologically elementary amenable (TEA) group of finite Hirsch length with no nontrivial locally elliptic normal closed subgroup  is solvable-by-compact. 
\end{corollary}

\vspace{.3cm}
The next result extends the first assertion of \cite[Theorem 2]{h}.

\vspace{.3cm}
\begin{proposition}  \label{level}
	If $G$ is topologically elementary amenable (TEA) of finite Hirsch length $n$, then $G$ is in $L\mathfrak X_{n+1}$. 
\end{proposition}
\begin{proof}
This is proved by induction on $n$. If $n=0$, then $G$ is compact and so in $\mathfrak X_1$. In general, if $G$ is compactly generated and non compact, by the argument used in the proof of the above lemma, it has normal closed subgroups $K \leq H\leq G$ such that $H$ is cocompact in $G$ and $H/K$ is a free abelian of positive rank, $h(K)<h(G)-1$ and $h(G/K) < h(G)$. By the induction hypothesis, $K\in L\mathfrak X_{n-2}$ and $G/K\in L\mathfrak X_{n-1}$. Then  $H\in (L\mathfrak X_{n-2})\mathfrak X_1=\mathfrak X_{n-1}$. Thus $G\in \mathfrak X_{n-1}\mathfrak X_{1}\subseteq  \mathfrak X_{n}$. If $G$ is not compactly generated, then all of its compactly generated subgroups are in $\mathfrak X_{n}$ by the above argument, and so $G\in L\mathfrak X_{n}$.   
\end{proof}

Note that there are finitely generated, torsion free,
elementary amenable groups which are not virtually solvable \cite[page 162]{h}. An argument based on \cite[Lemma 2.2]{fw} proves the next lemma (c.f. the proof of \cite[Proposition 3.7]{fw}).

\vspace{.3cm}  
\begin{lemma}  \label{ccpt}
	Let $G$ be a locally compact group and $\sigma$ be a family of cocompact closed subgroups of $G$. Then ${\rm asdim}(\Box_\sigma G) = \sup\{{\rm asdim}(\Box_\tau G): \tau\subseteq \sigma\ {\rm countable} \}$. 
\end{lemma}

\vspace{.3cm}
Note that, unless $\sigma$ is countable, the box space is not  metrizable \cite[Theorem 2.55]{r2}, and the left hand side of the equality in the above lemma should be understood in the general context of asymptotic dimension of (uncountable) families of metric spaces (the so called, total box space). But this does not lead to a different asymptotic dimension (by the subspace and union permanence of asymptotic dimension \cite[Theorems 6.2, 6.3]{g}; c.f. the argument in the proof of \cite[Proposition 3.7]{fw}).

A family $\sigma$  of cocompact subgroups of $G$ is called {\it separating} if for any non-empty relatively compact subset $F \subseteq \bar F\subseteq G \backslash \{1\}$, there is $G_\alpha\in \sigma$ with $F\subseteq G\backslash G_\alpha$. A group $G$ is {\it residually compact} (RC) if it has a separating family of cocompact subgroups, or equivalently, the family of all cocompact subgroups of $G$
is separating. This is essentially the same as the notion introduced in section \ref{Box} for compactly generated groups, in the following sense: a decreasing sequence $\sigma=(G_n)$ of cocompact closed subgroups of a compactly generated, locally compact, $\sigma$-compact, Hausdorff group $G$ has trivial intersection iff it is separating.  The family $\sigma$ is called {\it semiconjugacy-separating} if for any non-empty relatively compact subset $F \subseteq G$, there is
$G_\alpha\in\sigma$ with  $G_\alpha\cap g^GG_\alpha$ empty for every $g\in F \backslash \{1\}$, where $g^G$ is the
conjugacy class of $g$. As in \cite[Lemma 3.12]{fw}, this is equivalent to the condition that for any non-empty relatively compact subset $F \subseteq G$, there is
$G_\alpha\in\sigma$ with the quotient map $q_\alpha: G\to G/G_\alpha$ injective on $Fg$, for each $g$; which in turn is equivalent to the condition that for any non-empty relatively compact subset $F \subseteq G$, there is
$G_\alpha\in\sigma$ with $N(G_\alpha)\cap F$ either empty or $\{1\}$, where $N(H):=\cup_{g\in G} H^g$, for a subgroup $H$. In particular, every separating family of cocompact normal subgroups
is semi-conjugacy-separating. Also any family $\sigma$ of cocompact subgroups is semi-conjugacy-separating whenever  there exists a separating family $\tau$ of cocompact normal subgroups such that $\sigma$ is a refinement of $\tau$ (that is, for any $G_\beta \in\tau$ there is $G_\alpha\in\sigma$ with $G_\alpha\leq G_\beta$). Now a locally compact group $G$ is residually compact (RC) if and only if it has a semiconjugacy-separating family of cocompact subgroups. 

\vspace{.3cm}
The next result extends \cite[Proposition 3.16]{fw}. 

\vspace{.3cm}
\begin{lemma}  \label{rc3}
Let a residually compact (RC) group $G$ with a a plig metric $d$ act on itself by right multiplication. Let $\sigma$ be a semi-conjugacy-separating family of cocompact  subgroups of $G$.
Then ${\rm asdim}(G) \leq {\rm asdim}(\Box_{\sigma} G)$.	
\end{lemma}
\begin{proof}
As mentioned in the paragraph after Lemma \ref{ccpt}, we work with the total box space (which is always metrizable) if needed, but use the same notation for the sake of simplicity. 
Assume that asdim$(\Box_\sigma G) \leq n$ for some $n\in\mathbb Z^{+}$ and consider the metric on $\Box_\sigma G$ that restricts to the quotient metric of on quotients of $G$ by elements of $\sigma$. 
Given $ R > 0$, 
exists $S \geq R$ and $S$-bounded cover $\mathcal U$ of $\Box_\sigma G$ with multiplicity at most $n + 1$ and Lebesgue number at least $R$.
Choose $G_\alpha \in\sigma$ with $q_\alpha: G \to  G/G_\alpha$ 
injective on $B_{3S}(1)g = B_{3S}(g)$,  for any $g \in G$. By \cite[Lemma 3.15]{fw}, $q_\alpha$ is isometric on $B_S(g)$ for any $g$. Let $\mathcal U_\alpha$ consist of intersections of elements of $\mathcal U$ with $G/G_\alpha$ and lift it to a $G_\alpha$-invariant cover $\mathcal V$ of $G$ of multiplicity at most $n + 1$ and Lebesgue number at least
$R$, exactly as in the proof of \cite[Proposition 3.16]{fw}. This means that asdim$(G) \leq n$ and we are done.	    
\end{proof}

In the next lemma, for $R>0$ and a subset $Y$ of a metric space $(X,d)$, we use the notation $P (Y; R; d) := \{y \in X:  d(Y, y) \leq R\}$. 

\vspace{.3cm}
\begin{lemma}  \label{exact2}
	Let $1\rightarrow N\xrightarrow[]{\iota} G\xrightarrow[]{\pi} K\rightarrow 1$ be a short exact sequence of
	locally compact groups with $\iota$ (identified with) an inclusion. Let $H$ be a closed (but not
	necessarily normal) subgroup of $G$ and let $q : G \to G/H$ and $p : K \to
	K/\pi(H)$ be the corresponding quotient maps. Then
	
	($i$) for the $G$-action induced by $\pi$ on $K/\pi(H)$, there is a unique $G$-equivariant map $\rho: G/H \to K/\pi(H)$ with $\rho\circ q = p\circ\pi$,
	
	($ii$) for any $k\in K$, $\rho^{-1}(p(k))=
	q(gN)$, for any $g \in \pi^{-1}(k)$, and this $N$-invariant, 
	
	($iii$) $H$ is cocompact in $G$ iff $\pi(H)$ is cocompact in $K$ and $N \cap H$ is cocompact in $N$, 
	
	Furthermore, for any plig metric $d$ on $G$, 
	
	($iv$) the quotient metric $d^{'}$ on $K/\pi(H)$ is the same as the metric induced by $\rho$ from the quotient metric of $G/H$,
	
	($v$) for any $R > 0$,  $k \in K$ and $g \in\pi^{-1}(k)$, $\rho^{-1}(P (p(k); R; d^{'}))=q(P(N; R; d)g)$ contains the $R$-net $q(Ng)$ isometric to $N/(N \cap H^g)$ under the corresponding quotient metric.
		
\end{lemma}
\begin{proof}
($i$). The map $\rho(gH): = p(\pi(g))$ is well defined and $G$-equivariant, and unique by 
surjectivity of $q$.

($ii$). The first statement is clear and the second follows by normality of $N$.

($iii$). This follows from $(i)$ and $(ii)$ and continuity and properness of the action of $G$ on $G/H$ by left multiplication.

($iv$). This follows from $(i)$.

($v$) This follows from normality of $N$, the fact that $d$ is a plig metric and that the metrics on $K$ and $K/\pi(H)$ are quotient metrics (c.f., the proof of part $(5)$ in \cite[Lemma 4.1]{fw}). 	    
\end{proof}

The next lemma extends \cite[Proposition 4.3]{fw}. 

\vspace{.3cm}
\begin{lemma}  \label{exact3}
	Let $1\rightarrow N\xrightarrow[]{\iota} G\xrightarrow[]{\pi} K\rightarrow 1$ be a short exact sequence of
	locally compact groups, where $G$ has a plig metric. Let $\sigma$ be a family of cocompact subgroups of $G$. Consider the corresponding families $\sigma_1 :=\{N\cap G^g_\alpha: G_\alpha\in\sigma, g\in G\}$ and $\sigma_2:=\{\pi(G_\alpha): G_\alpha\in\sigma\}$ of cocompact subgroups of $N$ and $K$, respectively. Then
	$${\rm asdim}(\Box_\sigma G)\leq {\rm asdim}(\Box_{\sigma_1} N)+{\rm asdim}(\Box_{\sigma_2} K).$$	
\end{lemma}
\begin{proof}
	 The box families for which the asymptotic dimensions are calculated are quotients of the group by the cocompact subgroups in the family with the quotient metric. By the previous lemma, there are  1-Lipschitz
	 maps $\rho_\alpha : G/G_\alpha \to K/\pi(G_\alpha)$, and for any $R > 0$ and $y \in K/\pi(G_\alpha)$, the coarse fibre $\rho^{-1}(B_R(y)$ contains an $R$-net isometric to $N/(N \cap G^g_\alpha)$, for  $g\in (p\circ\pi)^{-1}(y)$. This means that we have a uniform coarse equivalence as in \cite[Theorem 3.8]{fw}, which gives the desired inequality.	    
\end{proof}

The following special case is of great importance.

\vspace{.3cm}
\begin{lemma}  \label{exact4}
	Let $1\rightarrow N\xrightarrow[]{\iota} G\xrightarrow[]{\pi} K\rightarrow 1$ be a short exact sequence of
	locally compact groups, where $G$ has a plig metric. Let $max$ be the family of all cocompact (normal) subgroups of $G$, and use the same notation for $N$ and $K$. Then 
	$${\rm asdim}(\Box_{max} G)\leq {\rm asdim}(\Box_{max} N)+{\rm asdim}(\Box_{max} K).$$	
\end{lemma}
\begin{proof}
	As in the above lemma, $max$ induces families $max_1$ of (normal) subgroups of $N$ and $max_2$ of (normal) subgroups of $K$. Since  
	${\rm asdim}(\Box_{max_1} N)\leq {\rm asdim}(\Box_{max} N)$, and the same for $K$ and $max_2$, the inequality follows from Lemma \ref{exact3}.    
\end{proof}

Now we are ready to prove the second main result of this section. 

\vspace{.3cm}
\begin{theorem}  \label{main4}
	If $G$ is topologically elementary amenable (TEA) and residually compact (RC), then for any semi-conjugacy-separating
	family $\sigma$ of cocompact subgroups of $G$, 
	$${\rm asdim}(G)\leq {\rm asdim}(\Box_\sigma G) \leq {\rm asdim}(\Box_{max} G) \leq h(G).$$
\end{theorem}
\begin{proof}
	The first inequality is proved in Lemma \ref{rc3}, and the second is trivial. For the last inequality, note that for group extensions by compact or cocompact normal subgroups, we have  inequality of Lemma \ref{exact4} for asymptotic dimensions, and  equality 
	$h(G) = h(N) + h(K)$ for the Hirsch length. Also we have the desired inequality when $G$ is in class $\mathfrak X_1$, by Proposition \ref{va}. Now the result follows by transfinite induction.
	\end{proof}

\begin{corollary}  \label{equality}
	In any of the following cases we have the equalities 
	$${\rm asdim}(G)= {\rm asdim}(\Box_\sigma G)=h(G),$$
	for any semi-conjugacy-separating
	family $\sigma$ of cocompact subgroups of $G$:
	
	$(i)$ $G$ is  residually compact and a semidirect product of a locally elliptic group by a polycyclic-by-compact group, 
	
	$(ii)$ $G$ is the wreath product of a compact abelian group with a polycyclic-by-compact group.  
\end{corollary}
\begin{proof}
Part $(i)$ follows from Lemma \ref{exact} and Proposition \ref{pcc}. For part $(ii)$,  observe that by an argument similar to the proof of \cite[Theorem 3.2]{gr}, $G$ is residually compact in this case. Now apply part $(i)$.	
	\end{proof}

\section{Totally disconnected groups} \label{td}

The class of totally disconnected locally compact second countable (t.d.l.c.s.c.) groups are recently subject of study as an important subclass of  Polish groups (separable and completely metrizable topological groups). The second countability assumption is not that essential, as every t.d.l.c.
group is a directed union of open subgroups which are s.c. modulo a compact normal
subgroup \cite[8.7]{hr}. The basic advantage of being t.d. is that by a classical result of van Dantzig  a t.d.l.c. group admits a local basis at identity of compact
open subgroups \cite[7.7]{hr} (showing that compact t.d.l.c. groups are profinite and so are compact open subgroups of  t.d.l.c. groups). Platonov has shown that a t.d.l.c.s.c. group, in which every finitely generated
subgroup is relatively compact, may be written as a countable increasing union of profinite
groups \cite{pl}. Also a residually discrete t.d.l.c.s.c. group may be written as a countable
increasing union of [SIN]-groups \cite{cm}.
These two results make the subclass of t.d.l.c.s.c. groups built from profinite and discrete groups specially interesting. One should ad to these evidences the fact that  kernel
of the adjoint representation of a $p$-adic Lie group belongs to this subclass \cite{gw}. 

Let us define the above subclass more rigorously. The class of {\it elementary groups} is the smallest class $\mathscr{E}$ of t.d.l.c.s.c. groups which contains all s.c. profinite groups and countable discrete groups, and is closed under taking topological group extensions by s.c. profinite groups and countable discrete
groups and increasing countable unions (of open subgroups) \cite[1.1]{we}. This class is known to be closed under group extensions, taking closed subgroups and quotients by closed normal subgroups, inverse limits, quasi-products, and local direct products, as long as the resulting group remains t.d.l.c.s.c. Also it passes from a cocompact subgroup to a t.d.l.c.s.c. group \cite[1.3]{we}. 

Each t.d.l.c.s.c. group $G$ has a unique maximal (resp. minimal) closed normal subgroup ${\rm Rad}_{\mathscr E}(G)$  (resp. ${\rm Res}_{\mathscr E}(G)$) whose quotient group is
elementary \cite[1.5]{we}.  This is called the {\it elementary radical} (resp. {\it residual}) of $G$. Also, a t.d.l.c.s.c. group with an open solvable subgroup, is known to be elementary \cite[1.11]{we}. Moreover, the  structural results of Wesolek reduces the study
of compactly generated t.d.l.c.s.c. groups  to that of elementary groups and topologically characteristically simple non-elementary groups \cite[1.6, 1.7]{we}. 
Since the topological version of the second isomorphism theorem is valid for t.d.l.c.s.c. groups \cite[2.2]{we} (c.f., \cite[5.33]{hr}), the intersection of two closed cocompact subgroups of a t.d.l.c.s.c. group is again cocompact. In particular, the family of all closed cocompact subgroups is filtering in the sense of Caprace and Monod, and it follows from  \cite[2.5]{cm}, that if a compactly generated t.d.l.c.s.c. group is residually compact (RC), then each RC-approximation of $G$ includes a subgroup which is compact-by-discrete (c.f., \cite[2.6]{we}). 

In a t.d.l.c.s.c. group $G$, the quasi-center $QZ(G)$ of $G$ consists of elements with open centralizer.
This is a (not necessarily closed) characteristic subgroup  \cite{bm}.  For the discrete residual Res$(G)$ of $G$, defined as the intersection of all open normal
subgroups, the quotient $G/{\rm Res}(G)$ is a countable increasing
union of [SIN]-groups \cite[2.7]{we}. The locally elliptic radical $\Lambda(G)$ of Platonov also plays an important role in t.d.l.c. groups and could be characterized as the $\mathcal C$-core , for the the collection $\mathcal C$ of all compact subgroups. Also, when $G$ is t.d.l.c.s.c., it has a   closed characteristic subgroup SIN$(G)$ containing $QZ(G)$, which is an increasing union of
compactly generated relatively open [SIN]-subgroups.

Elementary groups contain t.d.l.c.s.c. which are either [SIN]-groups,  solvable, locally elliptic t.d.l.c.s.c., or contain a compact open subgroup that has a dense quasi-center.
The class $\mathscr E$ of elementary t.d.l.c. groups could be constructed by a transfinite inductive process: let $\mathscr E_0$ be the trivial class and  $\mathscr E_1$ be the class of profinite or discrete groups. Suppose that $\mathscr E_\alpha$ is defined and put $\mathscr E_{\alpha+1}=(L\mathscr E_\alpha)\mathscr E_1$ (using the notation of the previous section). Finally, for a limit ordinal $\beta$, put $\mathscr E_\beta:=\cup_{\alpha<\beta} \mathscr E_\alpha$. Then 
$\mathscr E=\cup_{\alpha<\omega_1} \mathscr E_\alpha$ \cite[page 1396]{we}. We define the {\it construction rank} of $G\in \mathscr E$ by rk$(G):=\min\{\alpha: G \in\mathscr E_\alpha\}.$ If rk$(G)>1$ and $G$ is compactly generated, then $G$ is a group extension of either
a profinite group or a discrete group by an elementary group of strictly lower rank. Also if $G \in\mathscr E$, then so is any open subgroup $H$ and  rk$(H)\leq$ rk$(G)$. Similarly, if $G \in\mathscr E$, then so is the quotient by any compact subgroup $K$ and  rk$(G/K)\leq$ rk$(G)$ (note that the quotient $G/K$ remains elementary, even if $K$ is not compact \cite[3.12]{we}). If G is t.d.l.c.s.c. and H is a normal subgroup with $H, G/H \in \mathscr E$, then so is $G$ 
and rk$(G)\leq$ rk$(H) +$ rk$(G/H)$. Finally if $H$ is a closed subgroup, rk$(H)\leq$ 3rk$(G)$ \cite[3.2-3.8]{we} (note that in \cite{we}, the right hand sides of the last two inequalities are added with 1 and 3, but we avoid this by shifting the construction rank by 1). If $G$ is a locally solvable t.d.l.c.s.c. group, then $G$ is elementary with rk$(G)\leq4^n$, where $n$ is the solvable rank of $G$  (i.e.,  the minimum of the derived length of solvable compact open subgroups of $G$) \cite[8.1]{we}. In particular,  non-discrete, compactly generated,
t.d.l.c. topologically simple group are not locally solvable (c.f., \cite[Theorem 2.2]{wi}). As a concrete case of rank calculation, let $G$ be residually discrete, Res$(G) = \{1\}$, and $G$ is a countable increasing union of
open [SIN]-groups. A [SIN]-group is compact-by-discrete and so is elementary with rank at
most 2, thus rk$(G)$ is at most 3. 

For a class $\mathscr G$ of t.d.l.c.s.c. groups, the elementary closure $\mathscr{EG}$ of $\mathscr G$ is the smallest class of t.d.l.c.s.c. groups containing $\mathscr G$ and  all s.c. profinite groups and countable discrete groups, which is closed under group extensions of s.c. profinite groups, countable discrete groups,
and groups in $\mathscr G$ and closed under
countable increasing unions \cite[3.19]{we}. An example is the class of all l.c.s.c. $p$-adic Lie groups for all
primes $p$ \cite[Theorem 6.3]{we2}.

As far as we know, there is no result on the asymptotic dimension of 
 t.d.l.c.s.c. groups or elementary groups. In this section we give some results in this direction. 
 A {\it subquotient} of a topological group is a closed subgroup of a quotient of $G$ by a closed normal subgroup. This includes both closed subgroups and quotients by closed normal subgroups at the same time. When the group has a plig metric,  any subquotient has a canonical induced plig metric. If the asymptotic dimension of the group is finite, so is that of any subquotient, but a group with infinite asymptotic dimension may have many subquotients with finite asymptotic dimension. The key fact, we are going to show, is that for t.d.l.c.s.c. groups, the asymptotic dimension of the group is controlled by the asymptotic dimension of its discrete subquotients.

As we already mentioned, locally solvable groups are elementary with finite construction rank. It is known (and indeed shown by von Neumann himself in \cite{v}) that locally solvable groups are amenable. As these are not necessarily elementary amenable, the bounds in previous section could not be applied. Unfortunately, the above result also does not effectively apply in this case, as a locally solvable group is not finitely generated unless it is solvable. However, for a locally solvable group, the Hirsch length dominates the asymptotic dimension, but the former (which is the supremum of the Hirsch length of finitely generated solvable subgroups) may blow up. 

In the compactly generated case, we have the following characterization of groups with finite asymptotic dimension, which the  main result of this section. In the next result, by a ``discrete quotient'' of $G$ we mean the homogeneous space $G/H$ for an open subgroup $H$. When $H$ is also normal, $G/H$ is a discrete group.    

\vspace{.3cm}
\begin{theorem}  \label{main5}
	Let $G$ be a  t.d.l.c.s.c. compactly generated group and consider the following conditions:
	
	$(i)$ $G$ has finite asymptotic dimension,
	
	$(ii)$ $G$ has a discrete quotient with finite asymptotic dimension, 
	
	$(iii)$ every discrete quotient of $G$ has finite asymptotic dimension.
	
	Then $(i)\Rightarrow (ii)$ and $(iii)\Rightarrow  (i)$. If moreover, all open subgroups of $G$ are compact, then these conditions are equivalent. When $G$ is also a {\rm [SIN]}-group, in parts $(ii)$ and $(iii)$ quotients could be taken to be finitely generated quotient groups (i.e., quotients by open normal subgroups). 
\end{theorem}
\begin{proof}
	We employ van Dantzig theorem (ensuring that $G$ always have open compact subgroups; c.f., \cite[7.7]{hr}) and a topological version of Milnor-\v{S}varc lemma: If $X$  pseudo-metric space and a locally compact group $G$ acts on $X$ with orbit map $\iota_x: G\to X; g\mapsto gx$, so that the action is geometric, namely it is isometric, cobounded (i.e., $X$ could be covered with translate of a subset with finite diameter), locally bounded (i.e., for a bounded subset $B$ of $X$, each $g\in G$ has a neighborhood $V$ with $VB$ bounded in $X$) and metrically proper (i.e., $\iota_x^{-1}(B_r(x))$ is relatively compact in $G$, for each $x\in X$ and each $r>0$), then $G$ is $\sigma$-compact and $\iota_x$ is a quasi-isometry, for each $x\in X$. If moreover, $X$ is coarsely-connected (i.e., for some $c>0$, between each two points in $X$ one could insert finitely many points with distance of consecutive ones at most $c$), then $G$ is also compactly generated (c.f., \cite[4.C.5]{ch}). Now if $H\leq G$ is a compact open  subgroup, then the quotient space $G/H$ is paracompact. In particular, the translation action of $G$ on $G/H$ is geometric (c.f., \cite[4.C.8]{ch}), and so by Milnor-\v{S}varc lemma, asdim$(G)\leq$ asdim$(G/H)$. Indeed, in this case we have asdim$(G)=$ asdim$(G/H)$, as the orbit map is just the quotient map, which is surjective. This shows that $(i)$ implies $(ii)$ and $(iii)$ implies $(i)$. If all open subgroups of $G$ are compact, then all discrete quotients of $G$ have the form $G/H$, for some compact open subgroup $H$, thus  $(i)$ implies $(iii)$. Finally, when $G$ is   a  t.d.l.c.s.c. compactly generated {\rm [SIN]}-group, it has a local basis consisting of compact open normal subgroups, and in the above argument, $H$ could always be chosen to be normal as well. The rest works as above. 
\end{proof}

\begin{corollary}  \label{resdisc}
	Let $G$ be a  t.d.l.c.s.c. compactly generated group all of whose discrete quotients are finite. Then 
	${\rm asdim}(G)={\rm asdim}Res(G)<\infty.$ 
\end{corollary}
\begin{proof}
	If all discrete quotients of $G$ are finite, then by implication $(iii)\Rightarrow (i)$ of Theorem \ref{main5}, $G$ has finite asymptotic dimension. Also in this case, $Res(G)$ \cite{ms} is a cocompact, compactly generated subgroup of $G$  by \cite[Theorem F]{cm}, thus ${\rm asdim}(G)={\rm asdim}Res(G)$. 	
\end{proof}

Note that by a deep result of Caprace and  Monod, for a compactly generated totally disconnected locally compact
group all of whose discrete quotients are finite, the discrete residual admits no non-trivial discrete or compact quotient \cite[Corollary G]{cm}. However, this condition on all discrete quotients is a rather strong condition, for instance in this case if $G$ is moreover residually discrete, then it is forced to be compact. A less restrictive situation is when $G$ is a  t.d.l.c.s.c. compactly generated [SIN]-group. In this case, $G$ is residually discrete, i.e., $Res(G)=\{1\}$ \cite[Corollary 4.1]{cm}, however, one could only deduce that $G$ is compact-by-discrete \cite[Theorem 2.13]{gm}, so it may have infinite asymptotic dimension.

\bibliography{sn-bibliography}
%\bibliography{JAMS-paper}

\end{document}